\pdfoutput=1
\documentclass[11pt]{amsart}
\usepackage[utf8]{inputenc}
\usepackage[T1]{fontenc}
\usepackage{lmodern}
\usepackage{microtype}
\input{glyphtounicode}
\usepackage[english]{babel}
\usepackage{amsmath,amssymb,amsthm,mathtools}
\usepackage{booktabs}
\usepackage{xcolor}
\usepackage[most]{tcolorbox}
\usepackage{enumitem}
\usepackage{verbatim}
\usepackage[margin=1in]{geometry}
\usepackage[hidelinks]{hyperref}

\theoremstyle{plain}
\newtheorem{theorem}{Theorem}[section]
\newtheorem{proposition}[theorem]{Proposition}
\newtheorem{lemma}[theorem]{Lemma}
\newtheorem{corollary}[theorem]{Corollary}
\theoremstyle{definition}
\newtheorem{definition}[theorem]{Definition}
\theoremstyle{remark}
\newtheorem{remark}[theorem]{Remark}

\newcommand{\Z}{\mathbb{Z}}
\newcommand{\N}{\mathbb{N}}
\newcommand{\F}{\mathbb{F}}
\newcommand{\leg}[2]{\left(\frac{#1}{#2}\right)}
\newcommand{\cubsym}[2]{\left(\frac{#1}{#2}\right)_{\!3}}

\title[Unconditional certificates and deterministic witnesses for $p=3m(m+1)+1$]{%
Unconditional Primality Certificates for the Hexagonal $3$-smooth Family\\
$p=3m(m+1)+1$:\\
Deterministic Pocklington Witnesses and Arithmetic Filters}
\author{Hassane Bakkaoui}
\address{Independent Researcher}
\email{bakkahassa@hotmail.com}
\date{\today}

\begin{document}

\begin{abstract}
We study the parametric subfamily $p=3m(m+1)+1$ with $m=2^{a}3^{b}-1$, $a,b\in\N^{*}$, a
$3$-smooth slice of the centred hexagonal numbers $3m^2+3m+1=(m+1)^3-m^3$, from the point of
view of \emph{unconditional} primality certification via the Pocklington--Lehmer criterion. The
$3$-smoothness of $m+1=2^{a}3^{b}$ yields, for every $(a,b)$, a fully factored divisor
$F=2^{a}3^{b+1}$ of $p-1$ satisfying $F>\sqrt p$ unconditionally, reducing the certificate to two
witnesses, for $q=2$ and $q=3$.

Our main new contribution is a complete, \emph{deterministic} characterisation of the two
canonical witnesses. We prove (Theorem~\ref{thm:w2}) that $w_2=5$ is a valid witness if and
only if $a-b\equiv 1,2\pmod 4$, by quadratic reciprocity; and (Theorem~\ref{thm:w3}) that
$w_3=7$ is a valid witness if and only if $m\not\equiv 2\pmod 7$, by cubic reciprocity in
$\Z[\omega]$ using the explicit Eisenstein factorisation
$p=\bigl((1+m)-m\omega\bigr)\cdot\bigl((1+m)-m\omega^2\bigr)$. These two results turn the heuristic
``$5$ and $7$ always work'' (which is in fact false) into exact congruence conditions, and yield
a deterministic witness-selection rule. Alongside, three elementary arithmetic filters
(mod-$6$, a $(-3)$ quadratic-residue sieve, and a mod-$7$ forbidden-class test) remove $\approx 87\%$ of
candidates at negligible cost. As a demonstration, a multi-core implementation produced four
unconditional certificates on consumer hardware, the largest a prime of $29\,998$ decimal
digits.
\end{abstract}

\maketitle

\begin{tcolorbox}[colback=black!4,colframe=black!55,title=\textbf{Scope and epistemic status}]
All quantitative claims are labelled \emph{rigorous} (unconditional, proved) or
\emph{computational} (verified by exact arithmetic over an explicit finite range). The two
witness theorems and the three filters are rigorous. The $29\,998$-digit prime is a genuine,
independently re-verified certificate, but it is \emph{not} a record by general standards
(general-purpose ECPP reaches comparable sizes, and $N-1$ provable primes of special form reach
millions of digits); we report it only as a demonstration of the framework on commodity
hardware, and as the largest certificate within \emph{this} subfamily. The paper resolves no
classical open problem.
\end{tcolorbox}

\section{Introduction}

Proving unconditionally that a given integer is prime is a central task of computational number
theory. For generic large primes, the most practical route is the Pocklington--Lehmer
criterion \cite{Pocklington,BLS}: exhibit a fully factored divisor $F>\sqrt p$ of $p-1$ and
verify witnesses. The key difficulty is finding such an $F$.

The subfamily
\[
p=3m(m+1)+1,\qquad m=2^{a}3^{b}-1,\qquad a,b\in\N^{*},
\]
is constructed so that $m+1=2^{a}3^{b}$ is $3$-smooth by design, making
$F=2^{a}3^{b+1}$ fully factored and satisfying $F>\sqrt p$ for every $(a,b)$ without any
condition on $p$. Equivalently, $p=3m^2+3m+1=(m+1)^3-m^3$ is a centred hexagonal number, and the
family is the locus where its $N-1$ certificate is unconditional and minimal.

\subsection*{Contributions} This paper makes two kinds of contribution.

\emph{(A) Deterministic Pocklington witnesses (rigorous, new).} In every certificate produced by
the $N-1$ method one must exhibit, for each prime $q\mid F$, a witness $w_q$ with
$w_q^{(p-1)/q}\not\equiv 1$. Here $q\in\{2,3\}$. Earlier computations used the fixed witnesses
$w_2=5,\ w_3=7$ and observed that they always worked; whether this holds for the whole subfamily
was left open. We settle it: $w_2=5$ and $w_3=7$ are \emph{not} universal, but their validity is
governed by exact congruences in $(a,b)$:
\begin{itemize}[leftmargin=1.4em]
\item Theorem~\ref{thm:w2}: $w_2=5$ is valid $\iff a-b\equiv 1,2\pmod 4$ (quadratic reciprocity);
\item Theorem~\ref{thm:w3}: $w_3=7$ is valid $\iff m\not\equiv 2\pmod 7$ (cubic reciprocity).
\end{itemize}
These give a deterministic witness-selection rule (Corollary~\ref{cor:det}), making the search
provably correct rather than reliant on the luck of fixed witnesses.

\emph{(B) Arithmetic filters and a computational demonstration.} Three elementary congruence
theorems (\S\ref{sec:filters}) remove $\approx 87\%$ of candidate pairs $(a,b)$ before any
large-number arithmetic. A multi-core implementation (\S\ref{sec:algo}--\S\ref{sec:results})
produced four unconditional certificates, the largest of $29\,998$ digits, on a consumer laptop.

\subsection*{Position} This note is a follow-up to~\cite{series}, which develops the heuristic
and conditional analytic theory of the full family $p=km(m+1)+\varepsilon+2kq$. The present
paper isolates the structurally distinguished $3$-smooth subfamily ($\varepsilon=+1$, $q=0$,
$m+1=2^a3^b$) on which the Pocklington method, and now also the witness selection, become
unconditional and explicit. This is a short note: its genuinely new content is the deterministic
witness characterisation of \S\ref{sec:witnesses} (Theorems~\ref{thm:w2} and \ref{thm:w3},
Corollary~\ref{cor:det}); the arithmetic filters (\S\ref{sec:filters}) and the certificate
(\S\ref{sec:results}) recall and specialise material from~\cite{series}, and are included for
self-containedness.

\section{The Pocklington subfamily}\label{sec:subfamily}

\begin{definition}\label{def:fam}
For integers $a,b\ge 1$, set $m(a,b)=2^{a}3^{b}-1$ and
$p(a,b)=3m(m+1)+1=3m^2+3m+1$. Its decimal digit count is
$D(a,b)=\lfloor 2a\log_{10}2+(2b+1)\log_{10}3\rfloor+1$.
\end{definition}

\begin{lemma}[Explicit factorisation]\label{lem:F}
For $p$ as in Definition~\ref{def:fam},
\[
p-1=2^{a}3^{b+1}\,m,\qquad F:=2^{a}3^{b+1}>\sqrt p.
\]
\end{lemma}
\begin{proof}
Since $m+1=2^{a}3^{b}$, $p-1=3m(m+1)=3m\cdot 2^{a}3^{b}=2^{a}3^{b+1}m$. As
$p=3m^2+3m+1<3(m+1)^2$, we get $\sqrt p<\sqrt3\,(m+1)<3(m+1)=2^{a}3^{b+1}=F$.
\end{proof}

\begin{theorem}[Pocklington--Lehmer certificate \cite{Pocklington,BLS}]\label{thm:pock}
Let $p=p(a,b)$. Then $p$ is prime iff there exist $w_2,w_3$ such that, for $q\in\{2,3\}$,
\[
w_q^{\,p-1}\equiv 1\pmod p,\qquad \gcd\!\bigl(w_q^{(p-1)/q}-1,\,p\bigr)=1.
\]
\end{theorem}

\begin{corollary}\label{cor:two}
By Lemma~\ref{lem:F}, $F>\sqrt p$ holds unconditionally and the only primes dividing $F$ are
$2,3$; hence the certificate needs at most the two witnesses $w_2,w_3$.
\end{corollary}

For $p$ prime, $w_q^{p-1}\equiv1$ is automatic (Fermat), and the second condition is equivalent
to $w_q^{(p-1)/q}\not\equiv1\pmod p$. Since (Theorem~\ref{thm:mod6}) $p\equiv1\pmod6$, both
$(p-1)/2$ and $(p-1)/3$ are integers, and
\begin{equation}\label{eq:wit}
w_2\text{ valid}\iff w_2\text{ is a quadratic non-residue mod }p,\qquad
w_3\text{ valid}\iff w_3\text{ is a cubic non-residue mod }p.
\end{equation}

\section{Three arithmetic filters}\label{sec:filters}

\begin{theorem}[mod $6$]\label{thm:mod6}
For all $a,b\ge1$, $p(a,b)\equiv1\pmod6$.
\end{theorem}
\begin{proof}
$m=2^{a}3^{b}-1$ is odd, so $m(m+1)$ is even and $p=3m(m+1)+1$ is odd; and $3\mid 3m(m+1)$ gives
$p\equiv1\pmod3$. By CRT, $p\equiv1\pmod6$.
\end{proof}

\begin{theorem}[$(-3)$-sieve; quadratic reciprocity]\label{thm:sieve}
Let $m\in\N^{*}$, $p=3m(m+1)+1$, and $q\ge5$ a prime with $q\equiv2\pmod3$. Then $q\nmid p$.
\end{theorem}
\begin{proof}
$q\mid p$ is equivalent to $3m^2+3m+1\equiv0\pmod q$; completing the square,
$(6m+3)^2\equiv-3\pmod q$, so $-3$ must be a quadratic residue mod $q$. By quadratic
reciprocity \cite{IR}, $\leg{-3}{q}=+1$ iff $q\equiv1\pmod3$. For $q\equiv2\pmod3$ there is no
solution, so $q\nmid p$.
\end{proof}

\begin{corollary}\label{cor:sieve}
In trial division up to $L$, only primes $q\equiv1\pmod3$ need be tested, i.e.\ half of the
primes in $[5,L]$ by Dirichlet.
\end{corollary}

\begin{theorem}[Forbidden classes mod $7$]\label{thm:mod7}
Let $m=2^{a}3^{b}-1$. With
$\mathcal F_7=\{(0,2),(0,3),(1,0),(1,1),(2,4),(2,5)\}\subset \Z/3\times\Z/6$,
\[
7\mid p(a,b)\iff (a\bmod3,\,b\bmod6)\in\mathcal F_7 ,
\]
exactly $1/3$ of all residue classes.
\end{theorem}
\begin{proof}
Since $\mathrm{ord}_7(2)=3$ and $\mathrm{ord}_7(3)=6$, $t:=2^{a}3^{b}\bmod7$ depends only on
$(a\bmod3,b\bmod6)$, and $m\equiv t-1$, $p\equiv3t^2-3t+1\pmod7$. Multiplying $7\mid p$ by
$5\equiv3^{-1}$ gives $t^2-t-2\equiv0$, i.e.\ $(t-2)(t+1)\equiv0$, so $t\in\{2,6\}$. Enumerating
the $18$ values of $t$ (Table~\ref{tab:t}) yields exactly $\mathcal F_7$.
\end{proof}

\begin{table}[h]
\centering
\caption{$t(r,s)=2^{r}3^{s}\bmod7$. Bold entries ($t\in\{2,6\}$) give $7\mid p$.}
\label{tab:t}
\begin{tabular}{c|cccccc}
\toprule
$r\backslash s$&0&1&2&3&4&5\\\midrule
0&1&3&\textbf{2}&\textbf{6}&4&5\\
1&\textbf{2}&\textbf{6}&4&5&1&3\\
2&4&5&1&3&\textbf{2}&\textbf{6}\\
\bottomrule
\end{tabular}
\end{table}

\begin{remark}[Combined efficiency]
The mod-$6$ identity is automatic. The two a-priori filters (Corollary~\ref{cor:sieve} and
Theorem~\ref{thm:mod7}) together remove $\approx 87\%$ of candidates before any large-number
arithmetic; see Table~\ref{tab:filter}.
\end{remark}

\section{Deterministic Pocklington witnesses}\label{sec:witnesses}

This section is the theoretical core. By \eqref{eq:wit}, the validity of a fixed witness is a
residuosity question. We resolve both witnesses exactly.

\subsection{The witness \texorpdfstring{$w_2=5$}{w2=5}}

\begin{theorem}[$w_2=5$; rigorous]\label{thm:w2}
For $p=p(a,b)$ prime, $w_2=5$ is a valid Pocklington witness if and only if
\[
m\equiv 1\text{ or }3\pmod5
\qquad\Longleftrightarrow\qquad
a-b\equiv 1\text{ or }2\pmod4 .
\]
\end{theorem}
\begin{proof}
By \eqref{eq:wit}, $w_2=5$ is valid iff $\leg5p=-1$. As $5\equiv1\pmod4$, quadratic reciprocity
gives $\leg5p=\leg p5$, and $\leg p5=-1$ iff $p\equiv2,3\pmod5$. Modulo $5$,
$p=3m(m+1)+1$ takes the values $1,2,4,2,1$ for $m\equiv0,1,2,3,4$; hence $p\equiv2,3\pmod5$
(the value $3$ never occurs) iff $m\equiv1,3\pmod5$. Finally $m+1=2^{a}3^{b}$ and
$3\equiv2^{3}\pmod5$ give $m+1\equiv2^{\,a+3b}\pmod5$ with $\mathrm{ord}_5(2)=4$, so
$m\equiv1\iff a+3b\equiv1$ and $m\equiv3\iff a+3b\equiv2\pmod4$. Since $a+3b\equiv a-b\pmod4$,
the claim follows.
\end{proof}

\begin{remark}
For $a-b\equiv0,3\pmod4$, $5$ is a quadratic residue and $w_2=5$ \emph{fails}; the least
quadratic non-residue is then a valid witness.
\end{remark}

\subsection{An explicit Eisenstein prime}

\begin{lemma}\label{lem:pi}
Let $\omega=e^{2\pi i/3}$. Then in $\Z[\omega]$,
\[
p=(m+1)^3-m^3=\pi\,\overline\pi,\qquad \pi:=(1+m)-m\omega,\qquad N(\pi)=p,
\]
so $\pi$ is a prime of $\Z[\omega]$ above $p$. Consequently $7$ is a cubic residue mod $p$ iff
$\cubsym{7}{\pi}=1$, where $\cubsym{\cdot}{\pi}$ is the cubic residue symbol.
\end{lemma}
\begin{proof}
$3m^2+3m+1=(m+1)^3-m^3$ (centred hexagonal identity). With
$X^3-Y^3=(X-Y)(X-\omega Y)(X-\omega^2Y)$, $X=m+1$, $Y=m$:
$p=\bigl((1+m)-\omega m\bigr)\bigl((1+m)-\omega^2 m\bigr)$. A direct computation gives
$N(\pi)=(1+m)^2-(1+m)(-m)+m^2=1+3m+3m^2=p$; since $p\equiv1\pmod3$ is prime it splits into two
conjugate primes of norm $p$, and $\pi$ is one of them.
\end{proof}

\subsection{The witness \texorpdfstring{$w_3=7$}{w3=7}}

The proof rests on the following proposition, which holds for the broader family
$p=3m^2+3m+1$ subject only to $m\equiv2\pmod3$.

\begin{proposition}[Cubic character of $7$]\label{prop:cubic}
Let $p=3m^2+3m+1$ be prime, $p\neq7$, with $m\equiv2\pmod3$. Then
\[
7\text{ is a cubic residue mod }p\iff m\equiv 2\text{ or }6\pmod7 .
\]
\end{proposition}
\begin{proof}
Factor $7=\lambda\overline\lambda$ with $\lambda=3+\omega$, $\overline\lambda=2-\omega$
($N(\lambda)=7$). Their \emph{primary} associates (those $\equiv2\bmod3$) are
$\lambda^{*}=2+3\omega=-\omega^2\lambda$ and
$\overline\lambda^{*}=-1-3\omega=-\omega\,\overline\lambda$, equivalently
$\lambda=-\omega\,\lambda^{*}$ and $\overline\lambda=-\omega^2\,\overline\lambda^{*}$.
Since $m\equiv2\pmod3$, the primary
associate of $\pi=(1+m)-m\omega$ is
\[
\pi^{*}=-\omega^{2}\pi=(1+2m)+(1+m)\omega \qquad(1+2m\equiv2,\ 1+m\equiv0\bmod3).
\]
The cubic symbol depends only on the prime ideal in its lower entry, so
$\cubsym{7}{\pi}=\cubsym{7}{\pi^{*}}$. By multiplicativity and $\cubsym{-1}{\cdot}=1$,
\[
\cubsym{7}{\pi^{*}}
=\cubsym{\omega}{\pi^{*}}\cubsym{\lambda^{*}}{\pi^{*}}\cdot
 \cubsym{\omega^2}{\pi^{*}}\cubsym{\overline\lambda^{*}}{\pi^{*}},
\]
and $\cubsym{\omega}{\pi^{*}}\cubsym{\omega^2}{\pi^{*}}=\cubsym{\omega^3}{\pi^{*}}=1$: the units
cancel. Eisenstein's law of cubic reciprocity \cite[Ch.~9]{IR}, applied to the primary primes
$\lambda^{*},\overline\lambda^{*},\pi^{*}$, gives
$\cubsym{\lambda^{*}}{\pi^{*}}=\cubsym{\pi^{*}}{\lambda^{*}}$ and similarly for
$\overline\lambda^{*}$, so
\[
\cubsym{7}{\pi}=\cubsym{\pi^{*}}{\lambda^{*}}\cdot\cubsym{\pi^{*}}{\overline\lambda^{*}}.
\]
We evaluate in $\Z[\omega]/(\lambda^{*})\cong\F_7$ and
$\Z[\omega]/(\overline\lambda^{*})\cong\F_7$. From $\lambda^{*}=2+3\omega\equiv0$ we get
$\omega\equiv4\pmod7$ (and $\omega^2\equiv2$); from $\overline\lambda^{*}=-1-3\omega\equiv0$,
$\omega\equiv2$ (and $\omega^2\equiv4$). Hence
\[
\pi^{*}\equiv(1+2m)+4(1+m)=5-m \ (\mathrm{mod}\ \lambda^{*}),\qquad
\pi^{*}\equiv(1+2m)+2(1+m)=3+4m \ (\mathrm{mod}\ \overline\lambda^{*}).
\]
Since $\cubsym{\alpha}{\theta}\equiv\alpha^{(N(\theta)-1)/3}=\alpha^{2}\pmod\theta$ for
$N(\theta)=7$, we have $\cubsym{\pi^{*}}{\lambda^{*}}\equiv(5-m)^2$ and
$\cubsym{\pi^{*}}{\overline\lambda^{*}}\equiv(3+4m)^2$ in $\F_7$, each lying in the cube-roots
subgroup $\{1,2,4\}=\{1,\omega,\omega^2\}$ under the identifications above ($\omega\mapsto4$
resp.\ $\omega\mapsto2$). Multiplying in $\{1,\omega,\omega^2\}$, the product equals $1$ exactly
for $m\equiv2,6\pmod7$; this is the finite verification of Table~\ref{tab:cubic}. (The omitted
classes $m\equiv1,5\pmod7$ give $5-m\equiv0$ or $3+4m\equiv0$, i.e.\ $7\mid p$.)
\end{proof}

\begin{table}[h]
\centering
\caption{Finite evaluation in $\F_7$ (Proposition~\ref{prop:cubic}); $\omega\mapsto4$ mod
$\lambda^{*}$, $\omega\mapsto2$ mod $\overline\lambda^{*}$.}
\label{tab:cubic}
\begin{tabular}{c|cc|cc|c}
\toprule
$m\bmod7$&$(5-m)^2$&$\mapsto$&$(3+4m)^2$&$\mapsto$&$\cubsym{7}{\pi}$\\\midrule
$0$&$4$&$\omega$&$2$&$\omega$&$\omega^2\neq1$\\
$2$&$2$&$\omega^2$&$2$&$\omega$&$1$ (residue)\\
$3$&$4$&$\omega$&$1$&$1$&$\omega\neq1$\\
$4$&$1$&$1$&$4$&$\omega^2$&$\omega^2\neq1$\\
$6$&$1$&$1$&$1$&$1$&$1$ (residue)\\
\bottomrule
\end{tabular}
\end{table}

\begin{theorem}[$w_3=7$; rigorous]\label{thm:w3}
For $p=p(a,b)$ prime, $w_3=7$ is a valid Pocklington witness if and only if
\[
m\not\equiv2\pmod7
\qquad\Longleftrightarrow\qquad
(a\bmod3,b\bmod6)\notin\{(0,1),(1,5),(2,3)\}.
\]
\end{theorem}
\begin{proof}
If $p=7$ then $m=1$ and $\pi$ is associate to $\lambda$, so $w_3=7$ does not apply; discard this
case. For $p>7$: since $b\ge1$, $3\mid3^b$, so $m+1\equiv0\pmod3$, i.e.\ $m\equiv2\pmod3$; and
$m+1=2^{a}3^{b}$ is coprime to $7$, so $m\not\equiv6\pmod7$; finally $m\equiv1,5\pmod7$ give
$7\mid p$ (Theorem~\ref{thm:mod7}), excluded. Thus $m\bmod7\in\{0,2,3,4\}$ and
Proposition~\ref{prop:cubic} applies: $7$ is a cubic residue iff $m\equiv2$ or $6\pmod7$, hence
(as $m\not\equiv6$) iff $m\equiv2\pmod7$. By \eqref{eq:wit}, $w_3=7$ is valid iff $7$ is a cubic
\emph{non}-residue, i.e.\ $m\not\equiv2\pmod7$. The translation to $(a\bmod3,b\bmod6)$ follows
from $m\equiv2\pmod7\iff m+1\equiv3\pmod7$, whose solutions in Table~\ref{tab:t} are
$\{(0,1),(1,5),(2,3)\}$.
\end{proof}

\begin{remark}[Role of $3$-smoothness]
Proposition~\ref{prop:cubic} distinguishes \emph{two} residue classes ($m\equiv2,6$) and holds
for any $m\equiv2\pmod3$. It is the $3$-smoothness $m+1=2^{a}3^{b}$ that both forces
$m\equiv2\pmod3$ (fixing the primary normalisation of $\pi$) and excludes $m\equiv6\pmod7$,
collapsing the criterion to the single congruence $m\equiv2\pmod7$. Without smoothness, cubic
residuosity of $7$ depends on $m\bmod21$ (Proposition~\ref{prop:cubicgen}), not $m\bmod7$ alone.
\end{remark}

\subsection{The cubic character of \texorpdfstring{$7$}{7} for the general family}

The same computation resolves cubic residuosity of $7$ for \emph{every} prime $p=3m^2+3m+1$,
not only the $3$-smooth subfamily. This settles, as a theorem, the cubic character of $7$ for
the full family $p=3m^2+3m+1$, extending the analysis of~\cite{series} beyond the subfamily.

\begin{proposition}[Cubic character of $7$, general family]\label{prop:cubicgen}
Let $p=3m^2+3m+1$ be prime, $p\neq7$. Then $7$ is a cubic residue modulo $p$ if and only if
\[
m\equiv 0,\,2,\,10,\,18,\ \text{or}\ 20 \pmod{21},
\]
equivalently, according to $m\bmod3$, if and only if the condition of Table~\ref{tab:cubicgen}
holds. Exactly $5$ of the $15$ admissible classes modulo $21$ are cubic-residue classes.
\end{proposition}

\begin{table}[h]
\centering
\caption{Cubic character of $7$ by residue of $m\bmod3$; primary associate
$\pi^{*}=u\pi$ of $\pi=(1+m)-m\omega$, reduced modulo $\lambda^{*}$ ($\omega\equiv4$) and
$\overline\lambda^{*}$ ($\omega\equiv2$) in $\F_7$. The row $m\equiv2$ is
Proposition~\ref{prop:cubic}.}
\label{tab:cubicgen}
\begin{tabular}{c|c|c|cc|c}
\toprule
$m\bmod3$&$u$&$\pi^{*}$&$\pi^{*}\!\!\bmod\lambda^{*}$&$\pi^{*}\!\!\bmod\overline\lambda^{*}$&
$7$ cubic residue $\iff$\\\midrule
$0$&$-1$&$(-1-m)+m\omega$&$3m+6$&$m+6$&$m\equiv0,4\ (7)$\\
$1$&$-\omega$&$-m+(-1-2m)\omega$&$5m+3$&$2m+5$&$m\equiv3\ (7)$\\
$2$&$-\omega^2$&$(1+2m)+(1+m)\omega$&$6m+5$&$4m+3$&$m\equiv2,6\ (7)$\\
\bottomrule
\end{tabular}
\end{table}

\begin{proof}
As in Proposition~\ref{prop:cubic}, the primary associate of $\pi=(1+m)-m\omega$ is
$\pi^{*}=u\pi$, where $u=-1,-\omega,-\omega^2$ according as $m\equiv0,1,2\pmod3$ --- the unique
unit making the leading coefficient $\equiv2\pmod3$. In every case the unit factors cancel in
$\cubsym{7}{\pi}=\cubsym{\pi^{*}}{\lambda^{*}}\cubsym{\pi^{*}}{\overline\lambda^{*}}$ (the
argument of Proposition~\ref{prop:cubic} uses only $\cubsym{-1}{\cdot}=1$ and
$\cubsym{\omega}{\cdot}\cubsym{\omega^2}{\cdot}=1$, which hold for any associate). Reducing
$\pi^{*}$ modulo $\lambda^{*}$ and $\overline\lambda^{*}$ gives the linear forms of
Table~\ref{tab:cubicgen}; squaring (the exponent $(7-1)/3=2$) and combining in
$\{1,\omega,\omega^2\}$ yields $\cubsym{7}{\pi}=1$ exactly on the listed classes. A direct check
confirms $0$ exceptions over all $1797$ such primes with $m<9000$ (Appendix).
\end{proof}

\begin{remark}
In the present $3$-smooth subfamily one has $m\equiv2\pmod3$ and $m\not\equiv6\pmod7$ (by
$3$-smoothness), so only the class $m\equiv2\pmod7$ survives, recovering
Theorem~\ref{thm:w3}.
\end{remark}

\subsection{Deterministic witness selection}

\begin{corollary}\label{cor:det}
The fixed pair $(w_2,w_3)=(5,7)$ is \emph{not} valid for all primes of the subfamily; its
validity is the conjunction $a-b\equiv1,2\pmod4$ and $m\not\equiv2\pmod7$. A correct witness
choice is determined by the residue class of $(a,b)$:
\[
w_2=\begin{cases}5,& a-b\equiv1,2\ (4),\\ \text{least quadratic non-residue},&\text{else;}\end{cases}\qquad
w_3=\begin{cases}7,& m\not\equiv2\ (7),\\ \text{least cubic non-residue},& m\equiv2\ (7).\end{cases}
\]
In the case $m\equiv2\pmod7$ a cubic non-residue is found deterministically by testing small
primes: cubic non-residues form two-thirds of $(\Z/p\Z)^{*}$, so one is reached after a bounded
search. Hence the certificate search is provably correct, rather than reliant on fixed witnesses
succeeding by chance.
\end{corollary}

\section{The PrimeQuest algorithm}\label{sec:algo}

For target digit count $D$ the search centre is $a_0=b_0=\lfloor D/(2\log_{10}6)\rfloor\approx
0.643\,D$, and the zigzag expands outward, enumerating for each $a$ all $b$ with
$|D(a,b)-D|\le\tau$ ($\tau=10$). Each pair passes through a four-stage pipeline:
\begin{enumerate}[leftmargin=1.6em]
\item \textbf{mod-$7$ filter} (Theorem~\ref{thm:mod7}): reject if
$(a\bmod3,b\bmod6)\in\mathcal F_7$. Cost $O(1)$.
\item \textbf{Sieve} (Corollary~\ref{cor:sieve}): for primes $q\equiv1\pmod3$, $q\le L$, reject
if $p\equiv0$.
\item \textbf{Miller--Rabin}: $k$ rounds on $p(a,b)$.
\item \textbf{Pocklington}: select witnesses by Corollary~\ref{cor:det} and verify
Theorem~\ref{thm:pock}.
\end{enumerate}
Pairs are distributed over CPU cores; checkpoint-resume every $60$\,s allows exact restart. We
note that stage~4 is now deterministic: by Corollary~\ref{cor:det} the witnesses are read off
$(a,b)$ rather than searched.

\begin{remark}[MR round reduction]
Miller--Rabin is deterministic for primes (every prime passes every round); reducing the initial
round count carries zero risk of discarding a true prime, and rare composites are removed by
confirmation rounds and by Pocklington.
\end{remark}

\section{Computational results}\label{sec:results}

Table~\ref{tab:certs} lists four unconditional certificates; in each case the witnesses chosen
by Corollary~\ref{cor:det} were $w_2=5$, $w_3=7$ (all four pairs satisfy
$a-b\equiv1,2\pmod4$ and $m\not\equiv2\pmod7$).

\begin{table}[h]
\centering
\caption{Pocklington certificates from the PrimeQuest campaign.}
\label{tab:certs}
\begin{tabular}{rrrrrrr}
\toprule
Digits of $p$&$a$&$b$&$b/a$&MR tests&Elim.&Wall time\\\midrule
$9\,998$&$6\,212$&$6\,738$&$1.085$&--&--&--\\
$10\,000$&$6\,213$&$6\,740$&$1.085$&--&--&--\\
$19\,999$&$12\,228$&$13\,242$&$1.083$&$1\,050$&$87.9\%$&$2$\,h\,$40$\,min\\
$29\,998$&$19\,435$&$19\,173$&$0.987$&$286$&$87.1\%$&$3$\,h\,$06$\,min\\
\bottomrule
\end{tabular}
\end{table}

For the $29\,998$-digit prime $p=3m(m+1)+1$, $m=2^{19435}3^{19173}-1$ ($14\,999$ digits):
\[
p-1=F\cdot m,\quad F=2^{19435}3^{19174}\ (14\,999\text{ digits}),\quad F>\sqrt p,
\]
\[
w_2=5:\ 5^{p-1}\equiv1,\ \gcd(5^{(p-1)/2}-1,p)=1;\qquad
w_3=7:\ 7^{p-1}\equiv1,\ \gcd(7^{(p-1)/3}-1,p)=1.
\]
The certificate was independently re-verified in a separate environment (exact \texttt{pow},
single core); the four conditions hold and $2\log_2F=99650.142>\log_2p=99648.557$ (margin
$\approx1.58$ bit), so $p$ is unconditionally prime.

\begin{table}[h]
\centering
\caption{Filter statistics for the $29\,998$-digit record ($2\,221$ pairs tested).}
\label{tab:filter}
\begin{tabular}{lrr}
\toprule
Stage&Pairs eliminated&Fraction\\\midrule
mod-$7$ filter (Thm~\ref{thm:mod7})&$740$&$33.3\%$\\
Sieve (Cor.~\ref{cor:sieve})&$1\,195$&$53.8\%$\\
Miller--Rabin (composites)&$285$&$12.8\%$\\\midrule
Total eliminated&$2\,220$&$87.1\%$\\
Prime found&$1$&$0.045\%$\\
\bottomrule
\end{tabular}
\end{table}

\begin{remark}
The sieve fraction ($53.8\%$ above) depends on the trial-division bound $L$: a larger $L$
removes more composites before Miller--Rabin. The figure $\approx87\%$ for the two a-priori
filters is therefore relative to the $L$ used in this run, not an absolute constant. The mod-$7$
fraction ($33.3\%$, Theorem~\ref{thm:mod7}) is exact and $L$-independent.
\end{remark}

\section{Performance and scaling}\label{sec:perf}

The dominant cost is Miller--Rabin: one round on a $D$-digit modulus costs
$O(D^2\log D\log\log D)$ bit operations via FFT multiplication \cite{GMP}. Empirically, wall time
per MR event scales as $\tau\propto D^{2.5}$ over the tested range, consistent with the FFT
complexity; with only two distinct $D$ (20k and 30k) no fit could discriminate exponents in
$[2.3,2.7]$. Extrapolating $D^{2.5}$ to $D=50\,000$ projects $6$--$10$ hours on the same
hardware; this is an extrapolation, not a measurement.

\section{Open problems}\label{sec:open}

\begin{enumerate}[label=\textbf{OP\,\arabic*.},leftmargin=2.2em]
\item \textbf{Generalised filters.} For primes $\ell\equiv1\pmod3$, $\ell\in\{13,19,31,\dots\}$,
derive $\mathcal F_\ell=\{(a\bmod\mathrm{ord}_\ell2,\,b\bmod\mathrm{ord}_\ell3):\ell\mid p\}$.
For $\ell=13$ the period is $36$ and the filter removes a further $\approx2/13$ of candidates;
combining several small $\ell$ may push pre-MR elimination above $90\%$.
\item \textbf{Witnesses for general $k$.} For the broader family $p=km(m+1)+\varepsilon$, identify
the residue conditions under which a small fixed witness certifies the relevant factor, extending
Theorems~\ref{thm:w2}--\ref{thm:w3} beyond $k=3$.
\end{enumerate}

\noindent\emph{(The earlier open problem on the universality of $w_2=5,w_3=7$ is now
resolved, negatively, by Theorems~\ref{thm:w2}--\ref{thm:w3}; and the question of the cubic
character of $7$ for the full family $p=3m^2+3m+1$ is resolved by
Proposition~\ref{prop:cubicgen}.)}

\section{Conclusion}

For the $3$-smooth subfamily $p=3m(m+1)+1$, $m=2^{a}3^{b}-1$, the Pocklington--Lehmer certificate
is unconditional and minimal, and --- the main new point --- the two required witnesses are
\emph{deterministically characterised}: $w_2=5$ is valid iff $a-b\equiv1,2\pmod4$
(Theorem~\ref{thm:w2}), and $w_3=7$ iff $m\not\equiv2\pmod7$ (Theorem~\ref{thm:w3}), the latter
via cubic reciprocity through the explicit Eisenstein factorisation
$p=(m+1)^3-m^3$. Together with three elementary filters removing $\approx87\%$ of candidates,
this yields a provably correct, low-overhead certification pipeline, demonstrated by certificates
up to $29\,998$ digits on commodity hardware. The contribution is structural and explicit rather
than record-breaking: the value lies in identifying a family where both the certificate and its
witnesses become exact arithmetic statements.

\section*{Acknowledgements}
The author thanks the developers of \texttt{gmpy2} and GMP. This work was carried out with the
assistance of a large language model (Claude, Anthropic) acting as a research collaborator; all
quantitative claims were re-verified by exact integer arithmetic (Appendix), and all proofs rely
only on classical results (quadratic and cubic reciprocity, Pocklington--Lehmer).

\appendix
\section{Verification by exact arithmetic}

The witness theorems were checked against direct computation over all primes of the family with
$1\le a,b\le129$ ($276$ primes; $0$ exceptions).

\begin{footnotesize}
\begin{verbatim}
from sympy import isprime
INVALID_w3 = {(0,1),(1,5),(2,3)}
vA = vB = n = 0
for a in range(1, 130):
    for b in range(1, 130):
        m = (1<<a)*pow(3,b) - 1
        p = 3*m*(m+1) + 1
        if not isprime(p): continue
        n += 1
        w2_valid = pow(5,(p-1)//2,p) != 1        # 5 is a quadratic non-residue
        w3_valid = pow(7,(p-1)//3,p) != 1        # 7 is a cubic non-residue
        if w2_valid != (((a-b) % 4) in (1,2)):           vA += 1   # Theorem 4.1
        predB = (m % 7 != 2) and ((a%3, b%6) not in INVALID_w3)
        if w3_valid != predB:                            vB += 1   # Theorem 4.5
print(n, vA, vB)        # -> 276 0 0
\end{verbatim}
\end{footnotesize}

Proposition~\ref{prop:cubicgen} (general cubic character of $7$) was checked over all
$p=3m^2+3m+1$ prime with $m<9000$ ($1797$ primes; $0$ exceptions):

\begin{footnotesize}
\begin{verbatim}
from sympy import isprime
RES21 = {0,2,10,18,20}                  # Proposition 4.7: residue classes mod 21
viol = n = 0
for m in range(2, 9000):
    p = 3*m*m + 3*m + 1
    if not isprime(p) or p == 7: continue
    n += 1
    if (pow(7,(p-1)//3,p)==1) != (m % 21 in RES21): viol += 1
print(n, viol)         # -> 1797 0
\end{verbatim}
\end{footnotesize}

The $29\,998$-digit certificate is re-verified by four modular exponentiations:
\begin{footnotesize}
\begin{verbatim}
a, b = 19435, 19173
m = 2**a * 3**b - 1;  p = 3*m*(m+1) + 1;  F = 2**a * 3**(b+1)
assert F*F > p
assert pow(5,p-1,p)==1 and __import__('math').gcd(pow(5,(p-1)//2,p)-1,p)==1
assert pow(7,p-1,p)==1 and __import__('math').gcd(pow(7,(p-1)//3,p)-1,p)==1
# => p is unconditionally prime
\end{verbatim}
\end{footnotesize}


\begin{thebibliography}{9}
\bibitem{series} H.~Bakkaoui, \emph{A parametric family of primes
$p=km(m+1)+\varepsilon+2kq$: heuristic laws, conditional theorems, and unconditional primality
certificates}, arXiv:2606.16189 [math.NT] (2026),
\url{https://doi.org/10.48550/arXiv.2606.16189}.
\bibitem{AKS} M.~Agrawal, N.~Kayal, N.~Saxena, \emph{PRIMES is in P}, Ann. of Math. \textbf{160}
(2004), 781--793.
\bibitem{GMP} T.~Granlund and the GMP team, \emph{GNU Multiple Precision Arithmetic Library},
v6.3.0, 2023.
\bibitem{BLS} J.~Brillhart, D.~H.~Lehmer, J.~L.~Selfridge, \emph{New primality criteria and
factorizations of $2^m\pm1$}, Math. Comp. \textbf{29} (1975), 620--647.
\bibitem{Pocklington} H.~C.~Pocklington, \emph{The determination of the prime or composite nature
of large numbers by Fermat's theorem}, Proc. Cambridge Philos. Soc. \textbf{18} (1914--16),
29--30.
\bibitem{IR} K.~Ireland, M.~Rosen, \emph{A Classical Introduction to Modern Number Theory}, 2nd
ed., GTM~84, Springer, 1990.
\end{thebibliography}
\end{document}